\documentclass{article}

\usepackage{amsmath}
\usepackage{amssymb}
\usepackage[T1]{fontenc}
\usepackage[utf8]{inputenc}
\usepackage[english]{babel}
\usepackage[dvipsnames]{xcolor}
\usepackage{latexsym}
\usepackage{enumerate}
\usepackage{tabularx}
\usepackage{url}
\usepackage{multicol}
\usepackage{xcolor}
\usepackage{hyperref}
\usepackage[affil-it]{authblk}
\usepackage{tikz}
\usepackage{graphicx}
\usepackage{cleveref}
\usepackage{marvosym}
\usepackage{todonotes}
\usepackage{comment}
\usepackage{caption}
\usepackage{subcaption}
\usepackage{amsthm}
\usepackage{thmtools}
\usepackage{thm-restate}

\newtheorem{theorem}{Theorem}

\newtheorem{corollary}[theorem]{Corollary}

\newtheorem{lemma}[theorem]{Lemma}
\newtheorem{conjecture}[theorem]{Conjecture}

\declaretheoremstyle[headfont=\itshape,bodyfont=\itshape,headpunct={.},postheadspace=.5em]{claimstyle}
\declaretheorem[name=Claim,style=claimstyle]{claim}
\declaretheorem[name=Remark,style=claimstyle]{remark}
\crefname{claim}{claim}{claims}
\Crefname{Claim}{Claim}{Claims}
\crefname{conjecture}{conjecture}{conjectures}
\Crefname{Conjecture}{Conjecture}{Conjectures}
\crefname{thm}{theorem}{theorems}
\Crefname{thm}{Theorem}{Theorems}
\newcommand{\mydiamond}{\rotatebox[origin=c]{45}{$\vcenter{\hbox{$\Box$}}$}}
\newenvironment{subproof}[1]{\par\noindent \textit{Proof of #1.}\ }{\hfill \mydiamond\par\vspace{11pt}}
\renewenvironment{proof}{\par \noindent \textbf{Proof.} }{\hfill$\Box$\medskip}

\newcommand{\say}[1]{``#1''}

\begin{document}

\title{Oriented trees in $O(k \sqrt{k})$-chromatic digraphs,
\\a subquadratic bound for Burr's conjecture\thanks{This paper was partially supported by the ANR DIGRAPHS (ANR-19-CE48-0013). An earlier version was presented at the WG 2024 conference~\cite{confVersion}.}}

\author[]{Stéphane Bessy}
\author[]{Daniel Gonçalves}
\author[]{Amadeus Reinald}

\affil[]{{\small LIRMM, Univ Montpellier, CNRS, Montpellier, France.}}

\date{}

\maketitle

\begin{abstract}
    In 1980, Burr conjectured that every directed graph with chromatic number $2k-2$ contains any oriented tree of order $k$ as a subdigraph.
    Burr showed that chromatic number $(k-1)^2$ suffices, and this was later improved to $\frac{k^2}{2} - \frac{k}{2} + 1$ by Addario-Berry, Havet, Linhares-Sales, Reed and Thomassé.

    We prove the first subquadratic bound for Burr's conjecture: every digraph with chromatic number $\frac{8}{3} k \sqrt{k} + 7k$ contains any oriented tree of order $k$.
    Moreover, we provide improved bounds of $\sqrt{4/3} k \sqrt{k}+k/2$ for arborescences, and $\frac{(b+3)}{2}k$ for paths with $b$ blocks.
\end{abstract}

\section{Introduction}

In essence, many of the questions in structural graph theory ask for sufficient conditions for a (di)graph to contain some other given (di)graph.
A natural way to specify this question is to ask, when imposing that a (di)graph $G$ has chromatic number $c$, which (di)graphs are necessarily contained in $G$.
A well-known result by Erdős~\cite{erdosBook} states that there exist graphs with arbitrarily large girth and chromatic number.
Therefore, all (di)graphs one can hope to be necessarily contained in (di)graphs of chromatic number $c$ are (directed) forests.
This question is easily settled when the graph is undirected, as all $k$-chromatic graphs contain all trees of order $k$, which is tight by considering the clique $K_{k-1}$.
When asking for \emph{induced} subtrees instead, and with a (required) sparseness restriction of forbidding a fixed clique, this is the notorious Gyárfás–Sumner conjecture~\cite{gyarfasConjecture,sumnerConjecture}, which is still wide open.

In this paper, we investigate the directed variant of the question, asking which directed trees are contained as (non-induced) subdigraphs in digraphs with sufficiently large chromatic number.
We refer to the surveys of Havet~\cite{havet2013orientations} on the relation between oriented subgraphs and chromatic number, and Stein~\cite{steinSurvey2024} for general oriented tree containment.
A digraph is $c$\emph{-universal} for some value $c \geq 1$ if every digraph with chromatic number at least $c$ contains it as a subdigraph.
Since the chromatic number of a digraph is preserved when replacing digons with arcs, if $H$ is $c$-universal in the class of oriented graphs, it is also $c$-universal. In particular, we will only need to consider oriented graphs in our proofs.
Moreover, the result of~\cite{erdosBook} implies $H$ must be an oriented forest.
We are thus interested in upper bounds on the universality of oriented trees, for which Burr posed the following conjecture in 1980.
\begin{conjecture}[Burr~\cite{burrConjecture}]\label{conj:burr}
    Every oriented tree of order $k$ is $(2k-2)$-universal.
\end{conjecture}
Observe that this is the best one can ask for, as the out-star on $k$ vertices $S_k^+$ is not contained in the regular tournament of order $2k-3$.
In the same paper, Burr showed all oriented trees are $(k-1)^2$-universal.
This stood as the best bound for the problem for decades\footnote{{One of the reviewers brought to our attention a $O(k^{3/2})$ bound mentioned in~\cite{burrConjecture}, attributed to a personal communication with F.R.K. Chung. To the best of our knowledge, this result has not been published nor cited, in particular in the later 1982 note of Chung~\cite{chungNoteSubtrees} herself, where she cites the $(k-1)^2$ bound of Burr.}} until Addario-Berry, Havet, Linhares-Sales, Reed and Thomassé gave the following improvement in 2013.
\begin{theorem}[Addario-Berry et al. \cite{addario-berryOrientedTreesDigraphs2013}]\label{thm:addario-general}
    Every oriented tree of order $k$ is $(\frac{k^2}{2} - \frac{k}{2} + 1)$-universal.
\end{theorem}
\noindent
While this stood as the best general bound until now, improvements have been obtained for specific oriented trees, or by restricting the class of host digraphs.

Regarding results on specific oriented trees, a folklore degeneracy observation yields~\Cref{conj:burr} for out-stars and in-stars. Burr proved the conjecture in other particular cases, including for oriented stars~\cite{burrConjecture}, which was also obtained independently by Naia~\cite{naia2018large}.
One of the most general cases for which a \textsl{linear} bound is known is for antidirected trees~\cite{addario-berryOrientedTreesDigraphs2013}.
Still, a central class of oriented trees is oriented paths, for which Burr~\cite{burrConjecture} conjectured not only~\Cref{conj:burr}, but also $k$-universality (see also Chung~\cite{chungNoteSubtrees}).
In this vein, let us note that Havet and Dross~\cite{dross2021unavoidability} conjectured that every oriented tree of order $k$ with $\ell$ leaves is $(k + \ell - 1)$-universal.
In fact, the earliest universality result is the famous Gallai-Hasse-Roy-Vitaver Theorem \cite{gallaiConjecture,hasseConjecture,roy1967nombre,vitaver1962determination}, stating that \textsl{directed} paths of order $k$ are $k$-universal.
Towards general orientations of paths, a natural intermediate class is that of $b$-\emph{block paths}, those admitting a partition into $b$ maximal directed paths. In 2007, Addario-Berry et al. proved the following for $b=2$, improving on a result of El Sahili~\cite{elsahiliPathsTwoBlocks2004}.
\begin{theorem}[Addario-Berry et al. \cite{addario-berryPathsTwoBlocks2007}]\label{thm:addario-2blocks}
    Every oriented path of order $k \geq 4$ with $2$ blocks is $k$-universal.
\end{theorem}
\noindent
The best previous bound for $3$-block paths is $4k-4$ due to El Joubbeh \cite{eljoubbehThreeBlocksPaths2021}, and $4k-2$ for $4$-block paths by El Joubbeh and Ghazal \cite{eljoubbeh2024existencek}. In the same paper, the authors of~\cite{eljoubbeh2024existencek} show a bound of $O(b^2k)$ for paths with $b$ blocks.
Let us mention that path containment is also a very active subject with respect to other parameters, such as degree variants, see~\cite{stein2025antidirected,penev2025two,chen2025paths,skokan2025alternating,grzesik2025antidirectedpathsorientedgraphs}.

When considering the question over restricted classes of graphs, only a few cases are known.
For instance, acyclic digraphs \cite{addario-berryOrientedTreesDigraphs2013}, or more generally bikernel-perfect digraphs, satisfy the conjecture. This is the building block for the $k^2/2 + O(k)$ bound of \cite{addario-berryOrientedTreesDigraphs2013}, and will also be key in our proof of Theorem~\ref{thm:main}.
Perhaps the most fruitful restriction to date is Sumner's conjecture \cite{sumnerConjecture}, predating Burr's conjecture, and stating that every tournament of order $2k-2$ contains every oriented tree of order $k$. A bound of $3k-3$ was achieved by El Sahili \cite{elsahiliTreesTournaments2004}, using a technique of Havet and Thomassé~\cite{havet2000median}, and the conjecture was settled for sufficiently large $k$ by K{\"u}hn, Mycroft and Osthus~\cite{kuhn2011sumner}. See also~\cite{dross2021unavoidability,benford2022trees,BENFORD2025260} for variants of Sumner's conjecture.
Another interesting restriction is given by considering digraphs with sufficiently large chromatic number compared to their order, as shown by Naia~\cite{naia2022trees}. 

\paragraph{Our contributions}

Our main result is the first subquadratic bound for Burr's conjecture, improving the $O(k^2)$ bound of~\Cref{thm:addario-general} to $O(k^{3/2})$.
\begin{restatable}{thm}{main}\label{thm:main}
    Every oriented tree of order $k$ is {$(\frac{8}{3}k \sqrt{k} + 7k)$}-universal.
\end{restatable}
\noindent
We also show an improved bound of $(\sqrt{4/3} \cdot k \sqrt{k} + k/2)$ for arborescences in Theorem~\ref{thm:arbo}, whose (simpler) proof follows the same general scheme as~\Cref{thm:main}.
A direct adaptation of~\Cref{thm:main} also yields that oriented paths are $(k \sqrt{k} + O(k))$-universal.
Still, oriented paths remain a central case of Burr's conjecture, for which obtaining a lower exponent is an interesting open problem.

In that direction, our second result is a bound of $O(bk)$ for $b$-block paths, dropping the quadratic dependence on $b$ from~\cite{eljoubbeh2024existencek}.
\begin{restatable}{theorem}{bblockNew}\label{thm:bblock-new}
    Every oriented path of order $k$ with $b$ blocks is {$\frac{(b+3)}{2}k$}-universal.
\end{restatable}
\noindent
We also show a $(b-1)(k-2)+2$ bound in~\Cref{thm:bblock}, which settles~\Cref{conj:burr} for $3$-block paths.

\paragraph{Roadmap of our strategy}
A natural technique when trying to establish bounds for the universality of trees is to proceed by induction, usually starting from a star or a path.
A step of the induction goes as follows: given an oriented tree $T'$ which we know to be $c'$-universal, we consider a tree $T$ obtained by \say{gluing} simple subtrees to $T'$.
Here, gluing (or appending) a tree $T''$ to the tree $T'$ means considering the disjoint union of $T'$ and $T''$ and identifying one vertex of $T''$ with one vertex of $T'$.
We then derive a bound for the universality of $T$ in terms of $|T'|$, $c'$, and the subtree being glued.
It was shown in \cite{addario-berryOrientedTreesDigraphs2013} that augmenting $T'$ with $\ell$ out-neighbours appended to arbitrary vertices yields a tree $T$ that is $(c' + 2 |T| - 4)$-universal. This alone is enough to obtain a quadratic bound for the conjecture, where the worst cases are given by orientations of paths, and more generally trees with few leaves.

The cornerstone of our contributions is a complementary gluing lemma for oriented paths (\Cref{lem:oriented-gluing}), and its improvement for directed paths (\Cref{lem:directed-gluing}).
We leverage this in the following win-win strategy when deriving universality bounds for some $T$ by induction.
If $T$ has many leaves (more than $O(\sqrt{k})$), we apply the gluing lemma of \cite{addario-berryOrientedTreesDigraphs2013} to the tree obtained by removing the leaves of $T$.
Otherwise, we observe $T$ is decomposable into few oriented paths, allowing us to apply our path gluing lemma a controlled number of times. 

\paragraph{Further directions}

Our ability to control the growth of our universality bound for trees as $O(k \sqrt{k})$ stems from our construction of trees, and corresponding bounds, by gluing both paths and leaves.
In the border cases where a tree can be built by adding many leaves at each step, or by gluing few paths, better bounds can be obtained (see Proposition 8 in \cite{addario-berryOrientedTreesDigraphs2013} for leaves).
Nevertheless, each gluing operation to some tree $T$ increases the bound by at least $|T|$.
Thus, if a tree falls between these two categories, such as a star of degree $\sqrt{k}$ subdivided $\sqrt{k}$ times, our current techniques do not allow us to improve on the $O(k \sqrt{k})$ bound.
Adapting one of the gluing lemmas to get rid of this $O(|T|)$ cost may allow us to skew our win-win strategy towards that type of gluing, and yield a better bound.
Another way to improve our strategy is to manage to glue multiple paths in parallel, or even more general trees.
Note that our cost for gluing a path depends quadratically on its order, so we are still far from a linear bound for paths. Nevertheless, even an improvement in that direction would not directly yield bounds better than $O(k \sqrt{k})$ for the general case, as witnessed by the case of arborescences.

\paragraph{Structure of the paper}

In Section~\ref{sec:preliminaries}, we introduce various technical tools, and the gluing lemma for leaves. In Section~\ref{sec:arborescences}, we prove our gluing lemma for directed paths, deriving \Cref{thm:bblock} for $b$-block paths, and the $\sqrt{4/3} \cdot k \sqrt{k} + k/2$ bound for arborescences (Theorem~\ref{thm:arbo}).
In Section~\ref{sec:oriented}, we prove our gluing lemma for oriented paths and derive Theorem~\ref{thm:main}.
The proofs for arborescences and general oriented trees follow the same approach, with only the technical lemmas differing in nature.

\section{Preliminaries}\label{sec:preliminaries}

Throughout the paper, all graphs and digraphs we consider are finite and simple. In general, we follow standard terminology, which can be found in Bang-Jensen and Gutin~\cite{bang-jensenDigraphs2009}.
A {\it digon} in a digraph is a directed cycle of length 2.
Digraphs may contain digons, whereas {\it oriented graphs} do not.
Throughout the rest of the paper, we only consider oriented graphs, having already argued that universality bounds for oriented graphs translate directly to digraphs.
The {\it order} $|D|$ of a (di)graph $D$ is the size of its vertex set. 
The union $D \cup H$ of two (di)graphs is the (di)graph on vertex set $V(D) \cup V(H)$, with edge (arc) set the union of the edges (arcs) of $D$ and $H$.
A path with $b$ blocks is an oriented path whose arcs can be partitioned into $b$ maximal directed subpaths, or blocks.
When clear from the context, we may say \say{tree} or \say{path} to refer to an oriented tree or oriented path.
A (proper) {\it colouring} of a (di)graph $D$ is an assignment of colours to $V(D)$ such that any pair of adjacent vertices is coloured differently. Then, the {\it chromatic number} of $D$, denoted $\chi(D)$, is the minimum number of colours required in such a colouring.
A {\it directed acyclic graph}, \emph{DAG} for short, is a digraph containing no directed cycles.
Given a digraph $D$, an {\it in-kernel} (respectively {\it out-kernel}) of $D$ is an independent set $K$ that in-dominates (respectively out-dominates) $D$. That is, every vertex in $V(D) - K$ admits an out-neighbour (respectively in-neighbour) in $K$.
It is well-known that any DAG contains both an in-kernel and an out-kernel~\cite{neumann1944kernel}.

\subsection{Universality in Directed Acyclic Graphs}

The following lemma is a \say{rooted} version of the result of~\cite{addario-berryOrientedTreesDigraphs2013} that oriented trees on $k$ vertices are $k$-universal in DAGs.
For the sake of completeness, we provide its proof below.

\begin{lemma}\label{lem:bikernel-perfect} 
Let $T$ be an oriented tree of order $k$ rooted in some vertex $r$, and $D$ be a directed acyclic graph such that $\chi(D) \geq k$.
Then, $V(D)$ contains a subset $K$ with $\chi(D[K]) \leq k-1$, such that for every $x \in V(D) \setminus K$, $D$ contains a copy of $T$ where $r$ is identified to $x$ and whose all other vertices lie in $K$.
\end{lemma}
\begin{proof}
If $k=1$, the conclusion holds by choosing $K=\emptyset$.
Now, consider $k>1$ and a tree $T$ on $k$ vertices with root $r$. As $k>1$, there exists a leaf $l$ in $T$ different from $r$.
Denote $T\setminus \{l\}$ by $T'$ and let $l'$ be the (only) neighbour of $l$ in $T$.
We can assume that $l'$ is an in-neighbour of $l$, the other case is symmetric up to swapping in-kernels for out-kernels.
Let $K_0$ be an in-kernel of $D$, which exists since $D$ is acyclic.
Since $K_0$ is independent, $\chi(D\setminus K_0)\ge k-1$, for otherwise colouring $K_0$ with a single colour would yield $\chi(D) \leq k-1$. 
Then, the induction hypothesis provides some $K' \subseteq V(D\setminus K_0)$ for the tree $T'$.
Now, let us show that $K=K'\cup K_0$ satisfies the lemma for $T$ in $D$.

First, note that $\chi(D[K]) \leq \chi(D[K'])+ \chi(D[K_0]) \le (k-2)+1 = k-1$.
Moreover, for any vertex $x \in V(D) \setminus K$, there exists a copy of $T'$ in $D\setminus K_0$ where $r$ is identified to $x$ and the other vertices of $T'$ lie in $K'$.
{Then, as $K_0$ in-dominates the copy of $T'$ (belonging to $V(D) \setminus K_0$), $l'$ has an out-neighbour in $K_0$ to which we can identify $l$.}
\end{proof}

\subsection{Gluing leaves}

We introduce the lemma of \cite{addario-berryOrientedTreesDigraphs2013} for gluing leaves, the first example of obtaining universality bounds by induction.
Let $T$ be an oriented tree.
Let $u$ be a leaf of $T$, and let $u'$ be its unique neighbour in $T$.
If the arc between $u$ and $u'$ is oriented towards $u$, we say that $u$ is an {\it out-leaf} of $T$, otherwise, we say that $u$ is an {\it in-leaf} of $T$.
We let $L(T)$ denote the set of all leaves of $T$, and $Out(T)$ (resp. $In(T)$) denote the set of all of its out-leaves (resp. in-leaves).
\begin{lemma}[Addario-Berry et al. \cite{addario-berryOrientedTreesDigraphs2013}]
\label{lem:gluing-leaves-addario}
    Suppose $T$ is an oriented tree of order $k \geq 3$ other than $S_{k}^+$ or $S_{k}^-$.
    If $T - Out(T)$ (respectively $T - In(T)$) is $c$-universal, then $T$ is $(c+2k-4)$-universal.
\end{lemma}

The next corollary follows by performing two gluing operations in a row, one with out-leaves and one with in-leaves.
\begin{corollary}
\label{cor:gluing-leaves-addario-twice}
    Suppose $T$ is an oriented tree of order $k \geq 3$ other than $S_{k}^+$ or $S_{k}^-$.
    If $T-L(T)$ is $c$-universal, then $T$ is $(c+4k-9)$-universal. 
\end{corollary}
\begin{proof}
Assume without loss of generality that $T$ has at least one out-leaf. Otherwise, $T$ has at least one in-leaf and the proof is symmetric.
Let $T' = T-Out(T)$ and let us distinguish two cases.
First, suppose $T'$ is an in-star or an out-star of order at most $k-1$, and recall it is then $2(k-1)-2$-universal~\cite{burrConjecture}.
In turn, Lemma~\ref{lem:gluing-leaves-addario} yields that $T$ is $4k-8$-universal. By hypothesis, $T-L(T)$ is $c$-universal, so $c \geq 1$ by definition, implying $T$ is $(c+4k-9)$-universal.

We may now assume $T'$ is neither an out-star nor an in-star.
Then, we let $T'' = T'-In(T')$.
Observe $T''$ is a subtree of $T-L(T)$, so any digraph containing a copy of $T-L(T)$ contains a copy of $T''$.
Therefore, $T''$ is $c$-universal, and since $|T'| \leq k-1$, Lemma~\ref{lem:gluing-leaves-addario} yields that $T'$ is $(c+2k-6)$-universal.
Applying the same lemma again on $T$, with $|T|=k$, we conclude that $T$ is $(c+4k-10)$-universal, showing the result.
\end{proof}

\subsection{Splitting trees into paths}

The following decomposition of a tree into paths is rather classical, we show a proof of it for the sake of completeness.
Here, we state the lemma in terms of undirected trees.
We will later apply it to oriented trees ignoring arc directions, but consider the resulting paths with their initial orientation.
Let $T$ be an undirected tree, rooted in some vertex $r$.
A path $P=v_1,\dots ,v_{\ell}$ of $T$ is {\it descending} if $v_i$ is the parent of $v_{i+1}$ in $T$ for $i=1,\dots , \ell-1$. Then, we consider $v_1$ as the {\it beginning} of $P$.

\begin{lemma}
\label{lem:dec-tree-in-paths}
    Let $T$ be an undirected tree of order at least two, rooted in some $r \in V(T)$, with exactly $p$ leaves {distinct from $r$}.
    Then, there exist $p$ descending paths $P_1, \dots ,P_p$ of $T$ such that for every $i \in [1,p-1]$ the path $P_{i+1}$ and $T_i = \bigcup_{j=1}^{i}P_j$ intersect in the vertex beginning $P_{i+1}$, and $T = T_p = \bigcup_{j=1}^{p}P_j$.
\end{lemma}
\begin{proof}
Let us denote the leaves of $T$ {(other than $r$)} as $(l_1,\ldots,l_p)$.
First, we define $P_1$ as the path from the root $r$ of $T$ to $l_1$, which contains no other leaf distinct from $r$.
Then, assume we have defined descending paths $P_1,\dots ,P_i$ for $i < p$ satisfying the conditions of the lemma.
Assume also that $T_i=\bigcup_{j=1}^i P_j$ contains exactly the non-root leaves $(l_j)_{j \leq i}$, and in particular $l_{i+1} \notin T_i$.
We consider the descending path $P'_{i+1}$ from $r$ to $l_{i+1}$, noting $r \in T_i$, $l_{i+1} \notin T_i$, and $T_i$ is a subtree of $T$.
This allows us to define $P_{i+1}$ as the subpath of $P'_{i+1}$ beginning at the last vertex in $T_i$, and ending in $l_{i+1}$.
Then, $T_i$ and $P_{i+1}$ only intersect in the vertex beginning $P_{i+1}$, and $T_{i+1} = T_i \cup P_{i+1}$ has exactly $(l_j)_{j \leq i+1}$ as non-root leaves.
When this process ends, we have defined $p$ paths $P_1,\dots ,P_p$, such that $T_p=  \bigcup_{j=1}^{p} P_j =T$.
\end{proof}

\section{Arborescences and \texorpdfstring{$b$}{b}-block paths}
\label{sec:arborescences}

In this section, we tackle the universality of arborescences and paths with $b$ blocks.
Consider a tree $T'$ for which we have a universality bound $c'$.
Our goal is to obtain a universality bound for the tree $T$ obtained by gluing a directed path of length $\ell$ to some vertex in $T'$.
It turns out that gluing directed paths as such is considerably cheaper, in terms of the increase in the bound, than gluing arbitrary orientations of paths.
This gives us improved bounds for arborescences and $b$-block paths, and more generally for any oriented tree that decomposes into few directed paths.

\subsection{Gluing a directed path}

We first show a technical result for all \textsl{oriented} graphs, which provides us with all the structure required to prove the gluing lemma for directed paths.
\begin{lemma}\label{lem:directed-partition}
For every integer $\ell \ge 0$, and every oriented graph $D$, there exists a partition of $V(D)$ into sets $X,Y,Z$ such that:
\begin{itemize}
    \item[\normalfont (X)] Every vertex $x \in X$ admits distinct $y^-,y^+ \in Y$ such that $(y^-,x),(x,y^+) \in A(D)$, {and $D[Y]$ contains a directed path from $y^+$ to $y^-$.}
    \item[\normalfont (Y)] $D[Y]$ is a DAG, and every sink of $D[Y]$ is the beginning of a directed path of length $\ell$ whose remaining vertices are all in $Z$.
    \item[\normalfont (Z)] $\chi(D[Z])\le \ell$.
\end{itemize}
Symmetrically, there also exists a partition $(X,Y,Z)$ for $D$ satisfying $(X)$, $(Z)$ and the following in place of $(Y)$:
\begin{itemize}
\item[$(\overleftarrow{Y})$] $D[Y]$ is a DAG, and every source of $D[Y]$ is the endpoint of a directed path of length $\ell$ whose other vertices are all in $Z$.
\end{itemize}
\end{lemma}

\begin{proof}
    We emphasize that throughout the proof, $D=(V,A)$ is oriented.
    Note that to obtain the second partition, it suffices to reverse all the arcs of $D$ and consider the first partition on the resulting oriented graph.
    Therefore, we show how to obtain the first partition, proceeding by induction on $\ell$.
    
    For $\ell = 0$, let $Z = \emptyset$, let $Y$ be the vertex set of a maximal induced acyclic subdigraph of $D$, and let $X = V(D)\setminus Y$.
    Properties (Y) and (Z) hold trivially.
    For property (X), note that by maximality of $Y$, for any vertex $x\in X$ the graph $D[Y\cup\{x\}]$ contains a cycle passing through $x$.
    The neighbours $y^-,y^+ \in Y$ of $x$ on the cycle must be distinct, since $D$ is oriented, and the remainder of the cycle yields the directed path from $y^+$ to $y^-$, showing (X).

    Assume by induction that the lemma holds for a fixed $\ell \geq 0$, and let us denote $X',Y',Z'$ the corresponding subsets of $V$.
    We will build $X,Y,Z$ satisfying the lemma for $\ell+1$, and refer the reader to Figure~\ref{fig:gluing-directed} for a sketch of the construction.
    Informally, in trying to satisfy (Y), we transfer sinks $S'$ of $Y'$ to $Z'$, defining $Z$. Then, to satisfy (X), we transfer vertices from $X'$ to the remaining vertices of $Y'$, defining $Y$, after which $X$ consists of the vertices left in $X'$.
    We then argue that the sinks $S$ of $Y$ admit an out-neighbour in $S'$, allowing us to extend the directed path given by induction by an arc.

    Let $S' \subseteq Y'$ be the set of sinks in $D[Y']$, shown in darker red. Note that $S'$ is a stable set in $D[Y']$, and thus also in $D$.    
    Letting $Z= Z' \cup S'$, we then have $\chi(D[Z])\le \chi(D[Z']) + 1 \le \ell+1$, satisfying (Z). 

    We now turn to defining $Y$ and $X$.
    Consider the subdigraph $D[Y'\setminus S']$, depicted in lighter green, which is acyclic since $D[Y']$ is.
    We define $Y$ as a vertex set containing $Y'\setminus S'$, and which induces a maximal directed acyclic graph in $D[V \setminus Z]$. Then, we let $X = X' \setminus Y$.
    The set $Y$ can be obtained by starting with $Y=Y' \setminus S'$ and iteratively adding vertices of $X'$ (shown in darker green), while maintaining that $D[Y]$ is acyclic.
    After having considered all the vertices of $X'$, we are guaranteed that for every $x \in X$, there is a directed cycle going through $x$ in $D[Y\cup\{x\}]$.
    The neighbours $y^-,y^+ \in Y$ of $x$ on the cycle are thus distinct (since $D$ is oriented), of opposite orientations, and the rest of the cycle yields a directed path from $y^+$ to $y^-$, showing (X).

    We now turn to property (Y), letting $S \subseteq Y$ be the set of sinks in $D[Y]$.
    Recall $Y = (Y' \setminus S') \cup (Y \cap X')$, and let us first show that $S \subseteq Y' \setminus S'$, meaning none of the sinks of $D[Y]$ belong to $X'$. 
    Indeed, consider any $x' \in Y \cap X'$, and note the vertex $y'^+ \in Y'$ given by property (X') must belong to $Y'\setminus S'$.
    This is because of our assumption that $D$ is oriented, so $y'^+ \neq y'^-$, and the $y'^+y'^-$-directed path then ensures that $y'^+$ is not a sink of $D[Y']$.
    Hence, $x'$ has an out-neighbour $y'^+$ in $Y' \setminus S' \subseteq Y$, meaning it is not a sink of $D[Y]$.

    We have just shown that sinks $S$ of $D[Y]$ are sinks of $D[Y' \setminus S']$.
    Since these vertices are not sinks in $D[Y']$, they must have an out-neighbour in $S'$.
    By induction, all such out-neighbours are the beginning of a directed path of length $\ell$ whose remaining vertices are in $Z'$.
    Therefore, every vertex in $S$ is the beginning of some directed path of length $\ell+1$, whose remaining vertices are all in $Z = Z' \cup S'$, ensuring (Y).
\end{proof}

\begin{figure}
    \centering
    \includegraphics[scale=0.7]{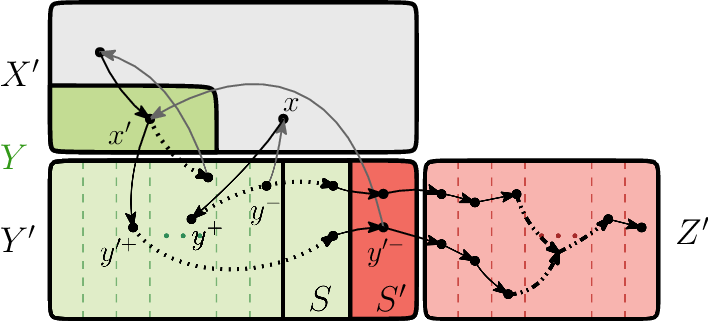}
    \caption{The partition of $V(D)$ into sets $X',Y',Z'$.
    The next step of the induction yields partition $X,Y,Z$, shown in grey, green, and red respectively.
    DAGs $D[Y']$ and $D[Y]$ are layered such that all arcs go from left to right.
    The sinks $S'$ of $D[Y']$ begin a directed path of length $\ell$ continuing in $Z'$, in dash-dotted. Then, the sinks $S$ of $D[Y]$ begin a directed path of length $\ell+1$ continuing in $Z$.
    }
    \label{fig:gluing-directed}
\end{figure}

We are now ready to prove the gluing lemma for directed paths.
\begin{lemma}\label{lem:directed-gluing}
    Let $T'$ be an oriented tree of order $k' \geq 1$, and let $T$ be obtained by gluing a directed path of length $\ell \geq 1$ (from any endpoint) to a vertex of $T'$.
    If $T'$ is $c'$-universal, then $T$ is {$c'+k'+2\ell-2$}-universal.
\end{lemma}
\begin{proof}
    Let $T'$ be any oriented tree of order $k'$, and consider an oriented graph $D=(V,A)$ with $\chi(D) \geq c'+k'+2\ell-2$.
    We show how to append a directed path of length $\ell$ to $T'$ by identifying the source of the path to some vertex of $T'$.
    The case is symmetric when appending a directed path from its sink, using the second partition given by~Lemma~\ref{lem:directed-partition} with property $(\overleftarrow{\text{Y}})$ instead of (Y).

    Let us apply Lemma~\ref{lem:directed-partition} to $D$ with {$\ell-1$} in place of $\ell$.
    We let $X,Y,Z$ be the corresponding partition of $V$, satisfying properties (X), (Y) and (Z).
    If $D[Y]$ has chromatic number at least $k'+\ell = |T|$, then by Lemma~\ref{lem:bikernel-perfect}, it contains the tree $T$.

    Otherwise, $D[Y]$ has chromatic number at most $k'+\ell-1$.
    Since $D[Z]$ has chromatic number at most $\ell-1$, we have that $D[X]$ has chromatic number at least {$c'+k'+2\ell-2 -(k'+\ell-1) - (\ell-1) = c'$}.
    By hypothesis $D[X]$ must contain a copy of $T'$, and we denote by $x$ the vertex on this copy corresponding to the vertex of $T'$ to which we wish to identify the source of the directed path.
    {Then, property (X) applied to $x$ gives distinct $y^-,y^+\in Y$ such that there is a path from $y^+$ to $y^-$ in $Y$. In particular, $y^+$ is not a sink, so
    we may follow a directed path of length at least one from $y^+$ towards a sink $s$ of $D[Y]$.}
    Then, property (Y) yields a path $P$ of length $\ell-1$ starting in $s$ and whose remaining vertices lie in $Z$.
    Now, $(x,y^+,\ldots,s)$ along with the path $P$ yield a directed path of length {at least $\ell+1$} starting in $x$, from which we preserve only the length $\ell$ prefix.
    Considering $T'$ in $X$ along with this prefix yields the desired copy of $T$ in $D$.
\end{proof}

Let us note a slight improvement, obtained whenever $\ell \geq 2$ by applying~\Cref{lem:directed-partition} with $\ell-2$ in place of $\ell$, and preserving the prefix of length $\ell$ (not $\ell -1$) at the end.
\begin{remark}
    If $\ell\ge 2$ in~\Cref{lem:directed-gluing}, $T$ is actually {$c'+k'+2\ell-3$}-universal.
\end{remark}

\subsection{\texorpdfstring{$b$}{b}-block paths}

With~\Cref{lem:directed-gluing} in hand, a first $O(bk)$ bound for $b$-block paths is straightforward.
\begin{restatable}{thm}{bblock}\label{thm:bblock}
    For $b\ge 2$, every oriented path of order $k \geq 4$ with $b$ blocks is {$(b-1)(k-2)+2$}-universal.
\end{restatable}
\begin{proof}
    We proceed by induction on $b$, with the case $b=2$ given by Theorem~\ref{thm:addario-2blocks}.
    Assume $P$ has $b\ge 3$ blocks, and consider a path $P'$ obtained by deleting one of the endblocks of $P$.
    If both choices for the deleted endblock result in $P'$ being of order $k' < 4$, since $P'$ has at least two blocks, we have $k' = 3$. Under this assumption, we must have $k=4$, and therefore $P$ is $4$-universal, as shown by Mortada, El-Sahili and El-Joubbeh~\cite{mortada2021paths}, satisfying the theorem.
    Otherwise, delete an endblock, say of length $\ell$, such that $P'$ is of order $k' = k- \ell \geq 4$.
    By induction hypothesis, since $P'$ has $b-1$ blocks, it is $(b-2)(k-\ell-2) +2$-universal.
    Now, $P$ is obtained from $P'$ by appending a directed path of length $\ell$.
    We apply Lemma~\ref{lem:directed-gluing}, which yields that $P$ is universal with the following bound:
    \begin{align*}
        &(b-2)(k-\ell-2) + 2 + (k-\ell) + 2\ell - 2 \\
        =& (b-1)(k - \ell -2) + 2 \ell + 2 \\
        \leq& (b-1)(k-2) + 2 \tag*{($b \geq 3$)}
    \end{align*}
\end{proof}


Note that \Cref{thm:bblock} confirms~\Cref{conj:burr} for $3$-block paths.
The same strategy with a refined analysis, comparing left to right and right to left gluing, allows us to improve the bound by a factor of $2$ asymptotically.
\bblockNew*
\begin{proof}
    The result holds for $b=1$ by the Gallai-Hasse-Roy-Vitaver theorem, and for $b=2$ either by~\Cref{thm:addario-2blocks} when $k \geq 4$, or by the oriented star bound~\cite{naia2018large} for $k =3$.
    Assume then we are given an oriented path $P$ with $b \geq 3$ blocks $P_1,\ldots,P_b$ of lengths $\ell_1,\ldots,\ell_b$, appearing in this order.
    We consider two ways of gluing these blocks, starting either from $P_1$ or from $P_b$.

    First consider the subpath of $P$ consisting of blocks $P_1,P_2$, which is $(\ell_1+\ell_2+1)$-universal whenever its order is at least four by~\Cref{thm:addario-2blocks}. Otherwise, its order is exactly three, yielding $4$-universality by~\cite{elsahiliPathsTwoBlocks2004}. In any case, it is $(\ell_1+\ell_2+2)$-universal.
    Then, we apply \Cref{lem:directed-gluing} successively to glue paths $(P_i)_i$ for increasing $i \in [3,b]$. This yields that $P$ is universal for the following value, where $(1+\sum_{j=1}^{i-1}\ell_j)$ corresponds to the order of the subpath with the first $i-1$ blocks.
    \begin{align*}
        & 2 + \ell_1 + \ell_2 + \sum_{i=3}^b \left[\left(1+\sum_{j=1}^{i-1} \ell_j \right) + 2 \ell_i - 2 \right] \\
        =& 2 - (b-2) + \ell_1 + \ell_2 + 2 \sum_{i=3}^b \ell_i + \sum_{i=3}^b  \sum_{j=1}^{i-1} \ell_j \\
        \leq &  4-b + \ell_1 +  2 \sum_{i=2}^b \ell_i + \sum_{i=2}^{b-1} \sum_{j=1}^{i} \ell_j \\
        \leq & 4-b +  2k + \sum_{i=1}^{b-1} \sum_{j=1}^{i} \ell_j
    \end{align*}
    The same reasoning replacing each index $u$ with $b-u+1$, that is, starting with the last two blocks and gluing for $i$ decreasing from $b-2$ to $1$ yields that $P$ is universal for the following value:
    \begin{align*}
        &4-b +  2k + \sum_{i=1}^{b-1} \sum_{j=1}^{i} \ell_{b-j+1}
    \end{align*}
    Now, summing the last sums of each bound yields:
    \begin{align*}
        & \sum_{i=1}^{b-1} \sum_{j=1}^{i} \ell_j + \sum_{i=1}^{b-1} \sum_{j=1}^{i} \ell_{b-j+1} \\
        =& \sum_{i=1}^{b-1} \sum_{j=1}^{i} \ell_j + \sum_{i=1}^{b-1} \sum_{j=b-i+1}^{b} \ell_{j} \\
        =& \sum_{i=1}^{b-1} \sum_{j=1}^{i} \ell_j + \sum_{i=2}^{b} \sum_{j=i}^{b} \ell_{j} \\
        =& \sum_{i=1}^{b-1} \sum_{j=1}^{i} \ell_j + \sum_{i=1}^{b-1} \sum_{j=i+1}^{b} \ell_{j} \\
        =& (b-1)(k-1)
    \end{align*}
    Therefore, one of the sums is at most $\frac{(b-1)(k-1)}{2}$, and adding the remaining $4 - b +  2k$ term yields:
    $$\frac{(b-1)(k-1)}{2} + 4 -b +2k = \frac{(b-1)}{2}k + 2k -\frac{(b-1)}{2} -b +4 \leq \frac{(b+3)}{2}k $$
    since $b \geq 3$, which shows that $P$ is $\frac{(b+3)}{2} k$-universal.
\end{proof}

Note that taking paths on $b+1$ blocks, made up of an arbitrarily large middle directed path, with alternating paths of length $b/2$ glued on each end, yields that the bound is asymptotically tight for this strategy.

\subsection{Growing arborescences}

We now move to the case of arborescences.
The proof is by induction, and follows the win-win argument that will also be used in the general case.
\begin{theorem}
\label{thm:arbo}
    Every arborescence of order $k$ is $(\sqrt{4/3} k \sqrt{k} + \frac k2)$-universal.
\end{theorem}
\begin{proof}
    Let $T$ be an out-arborescence of order $k$; the case of in-arborescences is symmetric.
    We set $g(k) = s \cdot k \sqrt{k} + k/2$ with $s = \sqrt{4/3}$ (and justify the value of $s$ later).
    The statement holds trivially for $k=1$, and whenever $T$ is an oriented star by $2k-2$-universality from~\cite{naia2018large}, so in particular for $k \leq 3$.
    Let us then assume $k \geq 4$, $T$ is not an oriented star, and that we have inductively shown the result for any out-arborescence of order less than $k$.
    Let $p$ be the number of out-leaves in $T$. We use two strategies  according to whether $p \leq s \sqrt{k}$.
    
    We first show a bound on the universality of $T$ that is suitable when $p$ is small (at most $s \sqrt{k}$).
    \begin{claim}\label{claim:pk-univ}
        $T$ is $pk$-universal.
    \end{claim}
    \begin{subproof}{the claim}
        Observe that $T$ contains exactly $p$ (undirected) leaves other than its source $r$ ($r$ may indeed be an in-leaf).
        Using Lemma~\ref{lem:dec-tree-in-paths} with $T$ rooted in $r$, we obtain $p$ descending paths $P_1,\dots,P_p$ of lengths $\ell_1,\ldots,\ell_p$.
        Consider these with their initial orientation, and as $T$ is an out-arborescence, each $P_i$ is a directed path.
        We let $T_i = \bigcup_{j=1}^i P_j$, such that $T=T_p$, and we know that $T_i$ and $P_{i+1}$ only intersect at the beginning of $P_{i+1}$.
        
        Now, we show by induction on $i$ that $T_i$ is $(i \cdot |T_i|)$-universal.
        For $i=1$, as $T_1$ is a directed path, it is $|T_1|$-universal by the Gallai-Hasse-Roy-Vitaver Theorem. 
        Assume the induction hypothesis holds for some $i<p$.
        Consider $T_{i+1}$, which is obtained from $T_i$ by appending the path $P_{i+1}$ rooted in their intersecting vertex.
        Then, Lemma~\ref{lem:directed-gluing} yields that $T_{i+1}$ is $(i \cdot |T_i|+|T_i|+2\ell_{i+1}-2)$-universal.
        As $|T_{i+1}|=|T_i|+\ell_{i+1}$, we obtain:
        \begin{align*}
            i \cdot |T_i|+|T_i|+2\ell_{i+1}-2 &\le (i-1) \cdot |T_i|+2(|T_i|+\ell_{i+1}) \\
            &\le (i-1) \cdot |T_{i+1}|+2(|T_{i+1}|) \\
             &=(i+1) \cdot |T_{i+1}|
        \end{align*}
        Therefore the induction holds, so $T=T_p$ is $pk$-universal.
    \end{subproof}

    If we fix the threshold $s$ such that we apply~\Cref{claim:pk-univ} whenever $p \leq s \sqrt{k}$, we obtain that $T$ is $s k^{3/2}$-universal in this case.
    In particular, the $k^{3/2}$ term of $g$ must be at least $s$, which is why we have fixed it this way.
    We now show a complementary bound that we will use whenever $p > s \sqrt{k}$, which will fix the actual value of $s$ (and also the additional $k/2$ term in $g$).

    Assume that $p \geq s \sqrt{k}$.
    Let $T' = T - Out(T)$, with $|T'| = k-p$.
    By induction, we know that $T'$ is $g(k-p)$-universal.
    Then, recall $T$ is not a star, so Lemma~\ref{lem:gluing-leaves-addario} yields that $T$ is $g(k-p) + 2k-4$-universal.
    To show that $T$ is $g(k)$-universal, we want $g(k) - g(k-p) \geq 2k-4$. Since $g$ is increasing, it suffices to show the following is at least $2k-4$ for the minimum value $p = s \sqrt{k}$:
    \begin{align*}
        g(k) - g(k-p)
        &= s\left(k^{3/2} - (k-p)^{3/2}\right) + \frac{p}{2} \\
        &= s \cdot \frac{3}{2}\int_{k-p}^{k}\sqrt{x} dx + \frac{p}{2} \\
        &\geq s \cdot \frac{3}{2} \left(\sqrt{k} + \sqrt{k-p}\right)\frac{p}{2} + \frac{p}{2} \tag*{(by concavity)} \\
        &= \frac{3}{4}s^{2}\sqrt{k}\left(\sqrt{k} + \sqrt{k - s\sqrt{k}}\right) + \frac{s}{2}\sqrt{k}  \tag*{$\left(p = s\sqrt{k}\right)$} \\
        &= \frac{3}{4}s^{2}\left(k + \sqrt{k^{2} - s\,k^{3/2}}\right) + \frac{s}{2}\sqrt{k}.
    \end{align*}
    We now use $s = \sqrt{4/3}$, implying $\frac{3}{4}s^{2} = 1$\footnote{Indeed, any lower value would fail to exceed $2k-4$ asymptotically.}.
    Now, the above equals $\sqrt{k^{2} - s\,k^{3/2}} + k + \frac{s}{2}\sqrt{k}$, and imposing it is at least $2k-4$ amounts to showing the following is non-negative.
    \begin{equation}\label{eq:arbo}
        \sqrt{k^{2} - s\,k^{3/2}} -k + \frac{s}{2}\sqrt{k} +4
    \end{equation}
    We claim this is non-negative for $k \geq 4$ (which we assumed).
    Indeed, for $k \in \{4,5\}$, it is positive because $- k + \frac{s}{2}\sqrt{k} + 4$ is.
    For $k \geq 6$, this second term is negative, and we equivalently want $\sqrt{k^{2} - s\,k^{3/2}} \geq k - \frac{s}{2}\sqrt{k} - 4$.
    Then, squaring both (non-negative) sides means this is equivalent to\footnote{{Note that the $-\frac{s}{2}\sqrt{k}$ term in the right-hand side of the first inequality stems from the $k/2$ term in $g$. Omitting it would fail to compensate for the $-sk^{3/2}$ term on the left hand side asymptotically.}}:
    \begin{align*}
        k^{2} - s\,k^{3/2}
        &\geq \left(k - \frac{s}{2}\sqrt{k} - 4\right)^{2}
        = k^{2} - s\,k^{3/2} + \frac{s^{2}}{4}k - 8k + 4s\sqrt{k} + 16 .
    \end{align*}
    Cancelling $k^{2} - s\,k^{3/2}$ and using $s^{2} = 4/3$, this reduces to $23k - 8\sqrt{3k} - 48 \geq 0$. This is positive for $k=4$, and has non-negative derivative for any $k \geq 4$, so it holds for any $k \geq 4$.
    This proves that $T$ is $g(k)$-universal whenever $p \geq \sqrt{4k/3}$.

    Having set $s=\sqrt{4/3}$, applying~\Cref{claim:pk-univ} whenever $p \leq \sqrt{4k/3}$, or the strategy above when $p > \sqrt{4k/3}$, completes the proof that $T$ is $g(k)$-universal.
\end{proof}

\section{Oriented trees}\label{sec:oriented}

In this section, we obtain universality bounds for general oriented trees.
As in the case of arborescences, we consider an oriented tree $T'$, for which we have a bound $c'$. Then, we derive a universality bound for the tree $T$ obtained by appending an oriented path $Q$ to $T'$.
We adapt the techniques of the previous section from directed paths to general oriented paths.

\subsection{Gluing an oriented path}

Given an oriented path $Q$, we obtain a {\it rooted oriented path} from $Q$ by choosing one of the endvertices of $Q$ as the {\it root} of $Q$.
Now, given a rooted oriented path $Q$, with root $r$, a digraph $D$ and a vertex $x$ of $D$, we say that $D$ contains a copy of $Q$ {\it starting at $x$} if $D$ contains a copy of $Q$ where $r$ is identified to $x$. 

The next two lemmas are the respective counterparts of Lemma~\ref{lem:directed-partition} and Lemma~\ref{lem:directed-gluing}. In both of those, the obtained bounds now depend quadratically on the length of the path being appended, whereas this dependence was linear in the last section.

\begin{lemma}\label{lem:oriented-partition}
Let $\ell \geq 0$ be an integer, and $Q$ a rooted oriented path of length $\ell$.
Then, for every {oriented graph} $D=(V,A)$, there exists a partition of $V$ into sets $X,Y,Z$ such that:
\begin{itemize}
    \item[\normalfont (X)] Every vertex $x \in X$ admits both an in-neighbour $y^-$ and an out-neighbour $y^+$ in $Y$,
    \item[\normalfont (Y)] $Y$ induces a directed acyclic graph, and for every vertex $y \in Y$, there exists a copy of $Q$ starting at $y$ and contained in $D[Y \cup Z]$,
    \item[\normalfont (Z)] $\chi(D[Z])\le \frac{\ell(\ell+1)}{2}$.
\end{itemize}
\end{lemma}

\begin{proof}
    We prove the result by induction on $\ell$.
    For $\ell = 0$, let $Z = \emptyset$, let $Y$ be (the vertices of) a vertex-wise maximal acyclic subdigraph of $D$, and $X = V\setminus Y$.
    Properties (Y) and (Z) hold trivially, while property (X) follows from the maximality of $D[Y]$ as an acyclic subdigraph of $D$.

    Assume by induction that the lemma holds for any rooted oriented path of length $\ell \geq 0$.
    To show it holds for paths of length $\ell+1$, let us consider any oriented path $Q$ of length $\ell+1$ rooted in $r$.
    Let $r'$ be the unique neighbour of $r$ in $Q$ and let $Q'$ be the oriented path $Q - r$ rooted in $r'$.
    Consider the partition $X',Y',Z' \subseteq V$ given by induction to satisfy the properties of the lemma for $Q'$.
    We will build a partition $X,Y,Z$ satisfying the properties for $Q$, and refer the reader to Figure~\ref{fig:gluing-oriented} for a sketch.
    Informally, in trying to satisfy (Y), we transfer the vertices of a suitably chosen set $K'\subseteq Y'$, from $Y'$ to $Z'$, defining $Z$.
    Then, to satisfy (X), we transfer vertices from $X'$ to the remaining vertices of $Y'$, defining $Y$, after which $X$ consists of the vertices left in $X'$.
    We show how to build copies of $Q$ starting in $Y$, either thanks to $K'$, or by appending a neighbour to some copy of $Q'$ obtained by induction.

    Let us start by defining $Z$, exhibiting a set $K' \subseteq Y'$, such that $Z = Z' \cup K'$.
    If $\chi(D[Y']) \leq \ell +1 $, we simply define $K' = Y'$.
    Otherwise, $D[Y']$ is an acyclic digraph such that $\chi(D[Y']) \geq \ell+2$.
    We may now apply Lemma~\ref{lem:bikernel-perfect} to $D[Y']$, with $Q$ as the oriented tree $T$ rooted in $r$ (note that $|Q|=\ell+2$).
    This gives us a subset $K' \subseteq Y'$ (shown in darker red), with $\chi(D[K']) \leq \ell+1$, and such that every {$y \in Y' \setminus K'$} is the beginning of some $Q$ whose remaining vertices belong to $K'$.
    We let $Z = Z' \cup K'$.
    We have shown that $\chi(D[K']) \leq \ell+1$ in both cases above, and by induction we have $\chi(D[Z']) \leq \frac{\ell(\ell+1)}{2}$.
    Colouring $D[K']$ and $D[Z']$ with a different set of colours gives $\chi(D[Z]) \leq \frac{\ell(\ell+1)}{2} + (\ell + 1) = \frac{(\ell+1)(\ell+2)}{2}$, satisfying (Z).

    We now turn to defining $Y$ and $X$.
    Consider the (possibly empty) subdigraph $D[Y' \setminus K']$, depicted in lighter green, which is acyclic since $D[Y']$ is.
    We define $Y$ as a vertex set containing $Y' \setminus K'$, and which induces a maximal directed acyclic graph in $D[V \setminus Z]$. Then, we let $X = X' \setminus Y$.
    The set $Y$ can be obtained by starting with $Y=Y' \setminus K'$ and iteratively adding vertices of $X'$ (shown in darker green), while maintaining that $D[Y]$ is acyclic.
    After having considered all the vertices of $X'$, we are guaranteed that for every $x \in X$, there is a cycle going through $x$ in $D[Y\cup\{x\}]$. The neighbours $y^-,y^+ \in Y$ of $x$ on the cycle must be distinct, since $D$ is oriented, which yields property (X).
    
    What is left to show is property (Y).
    Recall $Y = (Y' \setminus K') \cup (Y \cap X')$. In order to build $Q$ starting from $y \in Y$, we use two strategies according to whether $y \in (Y' \setminus K')$ or $y=x' \in (Y \cap X')$.
    In the first case, we must have $Y' \neq K'$, so $\chi(D[Y'])\ge \ell +2$, else we would have let $Y' = K'$ when defining $Z$.
    Moreover, we have already argued that $y$ must be the beginning of a copy of $Q$ whose remaining vertices lie in $K'$.
    Since $y \in Y$ and $K' \subseteq Z$, $Q$ is contained in $D[Y \cup Z]$, ensuring (Y) here.
    In the second case, we have $y=x' \in X'$ in particular, so $x'$ admits both an in-neighbour $y'^-$ and an out-neighbour $y'^+$ in $Y'$ by induction.
    Now, $Q$ is obtained from $Q'$ by appending either an out-neighbour, or an in-neighbour, respectively, to its root $r'$.
    We apply our induction hypothesis to consider a copy of $Q'$ starting at either $y'^-$, or $y'^+$, respectively.
    This copy lies in $Y' \cup Z'$, and therefore it cannot contain $x'$ since $X' \cap (Y' \cup Z') = \emptyset$ (see two examples for $y'^-$ in Figure~\ref{fig:gluing-oriented}). 
    Therefore, {$(y'^-,x')$, or $(x',y'^+)$} can be appended to the copy of $Q'$ to produce the required copy of $Q$, which shows (Y).
\end{proof}

We now move on to show the gluing lemma for oriented paths.
\begin{figure}
    \centering
    \includegraphics[scale=0.7]{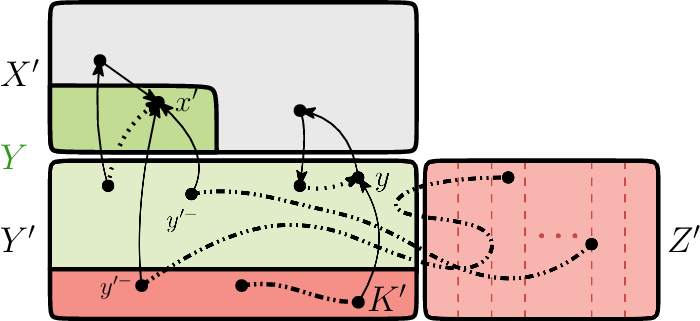}
    \caption{The partition of $V(D)$ into sets $X',Y',Z'$.
    The next step of the induction yields partition $X,Y,Z$, shown in grey, green, and red respectively.
    Vertices of $Y'$ begin a copy of $Q'$, in dash-dotted. Then, vertices of $Y$ begin a copy of $Q$, either in $K'$ (see $y$), or by appending an arc to some $Q'$ (see $x'$).}
    \label{fig:gluing-oriented}
\end{figure}

\begin{lemma}\label{lem:oriented-gluing}
    Let $T'$ be an oriented tree of order $k' \geq 1$, and let $T$ be obtained by appending a rooted oriented path of length $\ell \geq 1$ to any vertex of $T'$. If $T'$ is $c'$-universal, then $T$ is $(c' + k' + \frac{\ell(\ell+1)}{2} - 1)$-universal.
\end{lemma}
\begin{proof}
    Let us consider an oriented\footnote{Recall that universality within oriented graphs implies universality.} graph $D$ with chromatic number at least $c'+k'+\ell(\ell+1)/2-1$, and $Q$ the rooted oriented path of length $\ell$ we want to append to $T'$. We denote by $r$ the root of $Q$ and by $r'$ its unique neighbour in $Q$.
    Let $Q'$ be the oriented path of length $\ell-1$ obtained by deleting $r$ from $Q$ and rooted in $r'$.
    We start by applying Lemma~\ref{lem:oriented-partition} to $D$ with the rooted path $Q'$. We let $X,Y,Z$ be the corresponding partition of $V(D)$, satisfying properties (X), (Y) and (Z).
    If $D[Y]$ has chromatic number at least $k'+\ell = |T|$, since it is a directed acyclic graph, it contains the tree $T$ by Lemma~\ref{lem:bikernel-perfect}.

    Otherwise, $\chi(D[Y]) \leq k'+\ell-1$, and we also know $\chi(D[Z]) \leq (\ell-1)\ell/2$ by property (Z). Therefore, $D[X]$ has chromatic number at least $(c' + k' + \ell(\ell+1)/2 - 1) - (k'+\ell-1) - (\ell-1)\ell/2 = c'$.
    By hypothesis, $D[X]$ must contain a copy of $T'$, which we identify with $T'$ implicitly, and we denote by $x$ the vertex on this copy where $Q$ has to start to form $T$. 
    Then, property (X) ensures that $x$ has both an in-neighbour $y^-$ and an out-neighbour $y^+$ in $Y$.
    Both $y^-$ and $y^+$ are the beginning of a rooted copy of $Q'$ contained in $Y \cup Z$. Therefore, according to the orientation of the arc between $r$ and $r'$ in $Q$, we obtain a rooted copy of $Q$ starting in $x$ with either $(y^-,x)$ or {$(x,y^+)$}, with all remaining vertices in $Y \cup Z$.
    Except for $x$, this copy of $Q$ is disjoint from the copy of $T'$ containing $x$, yielding a copy of $T$ in $D$.
\end{proof}

\subsection{Growing oriented trees}

We are now ready to prove our universality bound for general oriented trees. 
\main*
\begin{proof}
    The proof follows the lines of~\Cref{thm:arbo}, using the oriented gluing lemmas, except for the case of few leaves, which requires an additional cutting of the paths.
    We set $f(k)= \frac83 k^{3/2} + 7k$, and show the result by induction on $k$.
    Let $T$ be an oriented tree of order $k$. The statement holds trivially for $k=1$, and whenever $T$ is an oriented star by $2k-2$-universality from~\cite{naia2018large}, so in particular for $k \leq 3$.
    We may thus assume $k \geq 4$, and that $T$ is not an oriented star.
    We use two strategies according to whether the number $p$ of leaves in $T$ is smaller or larger than $\sqrt{k}$.
    Note this corresponds to $s=1$ in~\Cref{thm:arbo}, but we do not optimize further for the sake of readability\footnote{A finer analysis could improve the leading constant by tweaking both the thresholds for the number of leaves, and for splitting paths lengths. We believe this can achieve $20/9$, but do not think that the current strategy can produce a term below $2$.}. 

    Assume first that $T$ has $p \geq \sqrt{k}$ leaves, and let $T' = T- L(T)$, which is of order $|T'| = k-p \leq k - \sqrt{k}$.
    By induction, we know that $T'$ is $f(k - p)$-universal, then Corollary~\ref{cor:gluing-leaves-addario-twice} yields that $T$ is $(f(k - p) + 4k-9)$-universal.
    Since $f$ is increasing, to show that $T$ is $f(k)$-universal, it suffices to verify $f(k) - f(k - \sqrt{k}) \geq 4k-9$.
    \begin{align*}
        f(k) - f(k - \sqrt{k}) &= \frac{8}{3}(k^{3/2}-(k-\sqrt{k})^{3/2}) + 7\sqrt{k} \\
        &= \frac{8}{3} \frac{3}{2} \int_{k-\sqrt{k}}^k \sqrt{x} dx + 7\sqrt{k} \\
        &\geq 4 \left(\sqrt{k}+\sqrt{k-\sqrt{k}}\right) \frac{\sqrt{k}}{2} + 7\sqrt{k} \tag*{(by concavity)}\\
        &= 2 \left(\sqrt{k} + \sqrt{k}\sqrt{1-1/\sqrt{k}} \right) \sqrt{k} + 7\sqrt{k} \\
        &\geq 2 \left(\sqrt{k} + \sqrt{k}(1-1/\sqrt{k}) \right) \sqrt{k} + 7\sqrt{k} \\
        &\geq 2(2 \sqrt{k} - 1) \sqrt{k} + 7\sqrt{k} = 4k + 5\sqrt{k} \\
        &\geq 4k-9
    \end{align*}
    This concludes this case, showing $T$ is $f(k)$-universal whenever $p \geq \sqrt{k}$.
    
    From now on, we may therefore assume that $T$ contains $p < \sqrt{k}$ leaves.
    Our goal is to split $T$ into oriented paths, which we use to build $T$ by using Lemma~\ref{lem:oriented-gluing}.
    Let us first apply Lemma~\ref{lem:dec-tree-in-paths} to the tree underlying $T$, rooted in an arbitrary non-leaf vertex $r$ (as $k \geq 4$). 
    This yields descending paths $Q_1,\dots ,Q_p$, which we root in their endpoint closest to $r$.
    We consider these rooted paths with their original orientations, such that $T = T_p = \bigcup_{j=1}^p Q_j$.
    Note then that $T$ can be obtained recursively by starting from $T_1 = Q_1$, and building $T_{i+1}$ by identifying the root of $Q_{i+1}$ to some vertex of $T_i$.
    
    At this point, the number of paths is bounded by $p$, but to avoid a quadratic dependence on $k$ from Lemma~\ref{lem:oriented-gluing}, we need paths of length $O(\sqrt{k})$.
    Set $\ell = \lceil \sqrt{k}/2 \rceil$.
    For $j \in [1,p]$, we split $Q_j$ into a minimal number of consecutive subpaths $(Q_{j,h})_h$, such that all have length at most $\ell$.
    Let $\ell_j$ be the length of $Q_j$, by minimality of the split, this yields at most $\ell_j / \ell +1$ paths $(Q_{j,h})_h$.
    Since $T$ contains $k-1 \leq k$ arcs, and the $(Q_{j,h})_{j,h}$ are arc-disjoint, their total number $m$ satisfies:
    \begin{align*}
        m \leq \sum_{j=1}^p (1 + \ell_j/ \ell) \leq p + k / \ell
    \end{align*}
    Now, relabel those paths as $(P_i)_{i \in [1,m]}$, ordered according to the $Q_j$ from which they stem, then with respect to their distance to the root of $Q_j$.
    We have $T=T_m= \bigcup_{i=1}^m P_i$, and with this order, $T$ can be built starting with $T_1=P_1$, and constructing $T_{i+1}$ by appending $P_{i+1}$ to $T_i$.
    
    We now derive values $(c_i)_i$ such that $T_i$ is $c_i$-universal. Since the order of $P_1=T_1$ is at most $\ell+1$, we have $c_1 \leq \ell(\ell+1)/2+1$ by~\cite{addario-berryOrientedTreesDigraphs2013}.
    Assuming a value $c_i$ for some $i$, Lemma~\ref{lem:oriented-gluing} yields that $T_{i+1}$ is $c_{i+1}$-universal for $c_{i+1} = c_i + k_i + \ell (\ell + 1)/2 - 1$.
    Now, since $T_i$ is obtained from $P_1$, of order at most $\ell+1$, by appending $(i-1)$ paths of length at most $\ell$, we have $k_i = |T_i| \leq i \ell + 1$.
    This yields:
    \begin{align*}
    c_{i+1} &\leq c_i + i \ell + 1 + \frac{\ell(\ell+1)}{2} - 1 \\
    &= c_i + i \ell +\frac{\ell(\ell+1)}{2}\\
    &\leq c_1 + \ell (\sum_{j=1}^i j) + \frac{i \ell (\ell+1)}{2} \textrm{~~~ by an immediate induction} \\
    &\leq \frac{1}{2}\ell(\ell+1) + 1 + \frac{1}{2} \ell i (\ell +i+2)
    \end{align*}
    Setting $i=m-1$ in the above, we obtain:
    \begin{align*}
        c_m &\leq \frac{1}{2}\ell(\ell+1) + 1 + \frac{1}{2} \ell (m-1) (\ell +m+1) = \frac{\ell}{2} \left( \ell+1 + (m-1)(\ell+m+1) \right) + 1 \\
        &=\frac{\ell}{2} \left(m \ell + m^2 \right) + 1 =\frac{\ell m}{2} \left(\ell + m \right) +1 
    \end{align*}
    Recall $\ell = \lceil \sqrt{k}/2 \rceil \leq \sqrt{k}/2+1$, and $m \leq k/\ell + p \leq 2\sqrt{k} + \sqrt{k} = 3\sqrt{k}$ since $p \leq \sqrt{k}$. Substituting these bounds in the above yields:
    \begin{align*}
        c_m &\leq \frac{3\sqrt{k}}{2}\left(\frac{\sqrt{k}}{2}+1\right)\left[\left(\frac{\sqrt{k}}{2}+1\right)+3\sqrt{k}\right] + 1
        = \frac{3\sqrt{k}}{2}\left(\frac{\sqrt{k}}{2}+1\right)\left(\frac{7\sqrt{k}}{2}+1\right) + 1 \\
        &= \frac{3\sqrt{k}}{2}\left(\frac{7k}{4}+4\sqrt{k}+1\right) + 1
        = \frac{21}{8}k^{3/2} + 6k + \frac{3}{2}\sqrt{k} + 1 \\
        &\leq \frac{8}{3}k^{3/2} + 7k = f(k) \tag*{(since $\frac{3}{2}\sqrt{k}+1 \leq k$ for $k \geq 4$)}
    \end{align*}
    This shows that $T$ is $f(k)$-universal, concluding the proof.
\end{proof} 

Observe that if $T$ is an oriented path in~\Cref{thm:main}, a better bound can be obtained by only applying the path gluing strategy. Indeed, splitting $T$ into $m=\sqrt{k}+O(1)$ pieces of length $\ell=\sqrt{k}+O(1)$, the bound of $(\ell m)(\ell+m)/2+1$ is asymptotically $k \sqrt{k} + O(k)$.

\subsection*{Acknowledgements}

We would like to thank the reviewers for their thorough reading and valuable suggestions, in particular for correcting the bound of~\Cref{lem:directed-gluing}.

\bibliographystyle{plainurl}
\bibliography{biblio-revision.bib}

\end{document}